\documentclass{amsart}

\usepackage{graphicx}

\newtheorem{theorem}{Theorem}[section]
\newtheorem{lemma}[theorem]{Lemma}

\theoremstyle{definition}
\newtheorem{definition}[theorem]{Definition}

\newcommand{\lp}{\dashv}
\newcommand{\rp}{\vdash}
\theoremstyle{remark}

\numberwithin{equation}{section}

\begin{document}

\title{Existentially closed Leibniz algebras and an embedding theorem}




\author{Chia Zargeh}
\address{Instituto de Matem\'atica e Estat\'istica, Universidade de S\~ao Paulo, S\~ao Paulo, Brazil.}
\email{chia.zargeh@ime.usp.br}
\address{Departamento de Matem\'atica, Universidade Federal da Bahia, Salvador, BA, Brazil}

\subjclass[2010]{17A32, 17A36, 16S15}

\date{2019}

\keywords{Existentially closed, Leibniz algebras, HNN-extension}

\begin{abstract}
In this paper we introduce the notion of existentially closed Leibniz algebras. Then we use HNN-extensions of Leibniz algebras in order to prove an embedding theorem.
\end{abstract}

\maketitle

\specialsection*{Introduction}
The notion of algebraically closed was originally introduced for groups in a short paper of  W. R. Scott \cite{S1}. A group $G$ is said to be \emph{algebraically closed} if every finite set of equations and inequations which is consistent with $G$ already has a solution in $G$. Scott applied the concept of algebraically closed in order to provide important embedding theorems stating that every countable group $G$ can be embedded in a countable algebraically closed group $H$.
There exists a rich literature on the properties of existentially closed groups and their applications, an interested reader can see \cite{H2} and \cite{L4}. Having considered the properties of closure of algebraic systems in both existentially and algebraically senses, we can claim that they are equivalent concepts for groups and Lie algebras. We recall that an algebraic system $A$ is existentially closed, if every consistent finite set of existential sentences with parameters from $A$, is satisfiable in $A$.
Shahryari in \cite{S2} used the concept of existentially closed groups and Lie algebras to prove some embedding theorems. For instance, Shahryari showed that any Lie algebra $L$ can be embedded in a simple Lie algebra in such a way for any non-zero elements $a$ and $b$, there is $x$ such that $[x,a]=b$. 
\\In this work, we introduce the concept of existentially closed for Leibniz algebras which are a non-antisymmetric generalization of Lie algebras introduced by Bloh \cite{B1} and Loday \cite{L3}. We follow Shahryari's approach to provide an embedding theorem analogous to Lie algebras. Our main tool in this work is the HNN-extension of Leibniz algebras. It is worth pointing out that the concept of HNN-extension has been recently spread to some generalizations of Lie algebras, including Leibniz algebras, Lie superalgebras and Hom-generalization of Lie algebras in a series of papers \cite{L1}, \cite{L2} and \cite{S3}.
\\The paper is organized as follows. Section \ref{section1} is devoted to preliminary tools and the concept of existentially closed for the case of Leibniz algebras. In Section \ref{section2}, we recall the concept of HNN-extensions of Leibniz algebras with more details.  In Section \ref{section3}, we provide a theorem on embeddability of any Leibniz algebra in a solvable Leibniz algebra.

\section{Existentially closed Leibniz algebras}\label{section1}
A right Leibniz algebra  is defined as a vector space with a bilinear multiplication such that the right multiplication is a derivation.
Indeed, a Leibniz algebra is a vector space $L$ over a field $\mathbb{K}$ with some bilinear product $[- , -]$ which satisfies the Leibniz identity
\[
[[x,y],z]=[[x,z],y]+[x,[y,z]].
\]
Let $I$ be a subspace of a Leibniz algebra $L$. Then $I$ is a subalgebra if $[I, I] \subset I$, a left (resp. right) ideal if $[L, I] \subset I $ (resp. $[I, L] \subset I$). $I$ is an ideal of $L$ if it is both a left ideal and a right ideal. The Leibniz algebra $Leib(X)$ is called free
Leibniz algebra with a set of generators $X$ if, for any Leibniz algebra $L$, an arbitrary map $X \to L$ can be extended to an algebra homomorphism
$Leib(X) \to L$. Then $X$ is called the set of free generators of $Leib(X)$.

One way of obtaining Leibniz algebras is to use a dialgebra $D$. This is a vector space
equipped with two bilinear associative products $\dashv$ and $\vdash$, and the laws 
\begin{align*}
x\lp y\lp z&= x\lp(y\rp z)\\
(x\rp y)\lp z&= x\rp(y\lp z)\\
(x\lp y)\rp z&= x\rp y\rp z.
\end{align*}
If we define $[x,y]=x\lp y - y \rp x$, then $(D,[-,-])$ becomes a Leibniz algebra. 
\begin{definition}
Leibniz algebra $L$ (with $[L, L] = I$) is said to be simple if the only ideals of $L$ are $\{0\}$, $I$, $L$.
\end{definition}
\begin{definition}
A Leibniz algebra $L$ is called solvable if there exists $n \in \mathbb{N}$ such that $L^{[n]}=0$, where $L^{[1]}=L$, $L^{[s+1]}=[L^{[s]},L^{[s]}]$ for $s \geq 1$.
\end{definition}
The following description of existentially closed Leibniz algebras is provided based on the common definition of existentially closed algebras (see, e.g., \cite{K1}).
Let $Leib$ be the class of Leibniz algebras over field ${k}$. For $A,B \in Leib$, the notation $A \ast B$ stands for free product of $A$ and $B$ in $Leib$. If $L \in Leib$, then $ \Phi \in L \ast Leib(X)$ can be considered as an $L$-valued function on $L$. An equation of the form $\Phi (x) =0$ is \emph{solvable} over $L$ if there exists an extension $\bar{L}$ of $L$ such that the equation has a solution in $\bar{L}$. In the case of finding such a solution in $L$ itself then $\Phi(x)=0$ is said to be solvable in $L$.
\begin{definition}
A Leibniz algebra $L$ is called existentially closed if every system of equations which is solvable over $L$ is solvable in $L$.
\end{definition}

\section{HNN-extensions of Leibniz algebras}\label{section2}
The Higman-Neumann-Neumann extensions (HNN-extensions) for groups was already introduced in \cite{H1}.  
If $A$ is a subgroup of a group $G$ and $t\in G$, then the mapping $a \mapsto t^{-1}at$ is an isomorphism between the two subgroups $A$ and $t^{-1}At$  of $G$.
  The HNN construction tries to reverse the viewpoint.
 For a group $G$ with an isomorphism $\phi$ between two of its subgroups $A$ and $B$,  $H$ is an extension of $G$
  with an element $t \in H$ such that $t^{-1}at= \phi(a)$ for every $a \in A$. The group $H$ is presented by
\[H=\langle G,t \mid t^{-1}at=\phi (a),  \ \text{for all} \  a\in A \rangle\]
and it implies that $G$ is embedded in $H$. The HNN-extension  of a group possesses an important position in algorithmic group theory which has been used for the proof of the embedding theorem, namely, that every countable group is embeddable into a group with two generators.

Ladra \textit{et al.} \cite{L2} studied the same construction for Leibniz algebras (as well as their
associative relatives, the so-called dialgebras) and proved that every Leibniz algebra embeds into any of its HNN-extensions. The main difference
between the construction of HNN-extension for groups and algebras is that
the concepts of subgroups and isomorphism are replaced by subalgebras
and derivation, respectively. In other words, the derivation map
defined on a subalgebra is used instead of isomorphism between subgroups.
In this section we recall the notion of HNN-extension of Leibniz algebras. We note that the HNN-extension for Leibniz algebras has been constructed corresponding to both derivation and anti-derivation maps. 

\begin{definition}
A derivation of Leibniz algebras is defined in a similar way to the derivation of Lie algebras, that is, a linear map $d: L \to L $ satisfying \[d([x,y]=[d(x),y]+[x,d(y)],\]
for all $x,y \in L$.
\end{definition}
\begin{definition}
An \emph{anti-derivation} of Leibniz algebras is defined as a linear map $d^{\prime} \colon L \to L$ such that
\[
d^{\prime}([x,y])=[d^{\prime}(x),y]-[d^{\prime}(y),x]
\]
for $x,y \in L$.
\end{definition}

\subsection*{HNN-extensions of Leibniz algebras}
Let $L$ be a right Leibniz algebra and $A$ be a subalgebra. We assume that the derivation $d$ and anti-derivation $d^{\prime}$ are defined on a subalgebra $A$ instead of the whole $L$.  The HNN-extensions of the Leibniz algebra $L$ corresponding to the derivation $d$ and the anti-derivation $d^{\prime}$ are defined as follows, respectively: 
\begin{equation}\label{hnnderivation}
L_d^{\ast} := \langle L, t: d(a)=[a,t], \ a\in A \rangle,    
\end{equation}

and 
\begin{equation}\label{hnnantiderivation}
L_{d^{\prime}}^{\ast} := \langle L, t: d^{\prime}(a)=[t,a], \ a\in A \rangle.  
\end{equation}
Here $t$ is a new symbol not belonging to $L$. By this, a new generating letter $t$ is added to any presentation of $L$. There are two special cases of HNN-extensions of Leibniz algebras.
\begin{itemize}
  \item If $A=L$, then $d$ is a derivation of $L$ and $L_d^{\ast}$ is then the semidirect product of $L$ with a one-dimensional Leibniz algebra which acts on $L$ via $d$.
  \item If $A=0$, then $L_d^{\ast}$ is the free product  of $L$ with a one-dimensional  Leibniz algebra.
\end{itemize} 

If a Leibniz algebra $L$ has a presentation $\langle X \mid S \rangle$ in the class of Leibniz algebras, then we have a presentation $\langle X \mid S^{(-)} \rangle$ in the class of dialgebras, where $S^{(-)}$ is the set of polynomials obtained from $S$ by changing the brackets as 
\[ [x,y]  = x \dashv y - y \vdash x .\] 
Indeed, for any Leibniz algebra $L$, there exists a unique universal enveloping dialgebra $U(L)$. The next theorem can be considered as one of the applications of HNN-extensions of Leibniz algebras. 

\begin{theorem} \cite{L2} \label{hnn-extension}
Every Leibniz algebra embeds into its HNN-extension.
\end{theorem}
The proof of the above theorem is based on the validity of Poincare-Birkhof-Witt theorem for Leibniz algebras which justifies the relation between construction of HNN-extension for dialgebras and HNN-extensions for Leibniz algebras. For an extensive proof see \cite{L2}.

\section{Embedding theorem}\label{section3}
In this section we provide an embedding theorem. To this end, we use the following lemma and theorem proved by Shahryari in \cite{S2}. In fact, both lemma and theorem  can be considered for an arbitrary non-associative algebra. 
\begin{lemma}\cite{S2}.
Let $\mathfrak{V}$ be an inductive class of algebras over a field $K$. Suppose $\mathfrak{V}$ is closed under subalgebra and $L \in \mathfrak{V}$. Then there exists an algebra $H \in \mathfrak{V}$ containing $L$ such that its dimension is at most 
\[ max \{\aleph_0, dim L, |K| \}. \]
Further, for any system $S$ of equations and in-equations over $L$, there exists the either of the following assertions:
\begin{itemize}
    \item [1-] $S$ has a solution in $H$
    \item [2-]  For any extension $H \subset E \in \mathfrak{V}$, the system $S$ has no solution in $E$.
\end{itemize}
\end{lemma}
\begin{proof}
 It is assumed that $X$ is a countable set of variables and 
 \[ \eta= max \{\aleph_{0}, dim L, |K| \}. \]
 Any equation over $L$ consists of finitely many elements of $L$ and $X$. The number of system of equations and in-equations over $L$ is $|L \cup K|$ and denoted by ${\kappa}$. Note that $|L|= max \{ dim L, |K| \}$, hence $\kappa=|L|+\aleph_{0}=\eta$. Let us consider a well-ordering in the set of all systems as $\{S_{\alpha}\}_{\alpha}$, using ordinals $0 \leq \alpha \leq \kappa$. Suppose $L_{0}=L$. 
 \\For any $0 \leq \gamma \leq \alpha$, the algebra $L_{\gamma} \in \mathfrak{V}$ is defined in such a way that $|L_{\gamma}|\leq \kappa$ and 
 \[ \beta \leq \gamma \Rightarrow L_{\beta} \subset L_{\gamma}. \]
 We put
 \[ E_{\alpha}= \bigcup_{\gamma \leq \alpha} L_{\gamma},\]
 so $E_{\alpha} \in \mathfrak{V}$ and, further, $|E_{\alpha}|\leq \alpha|L_{\gamma}|\leq {\kappa}^{2}=\kappa$. Suppose that $S_{\alpha}$ has no solution in any extension of $E_{\alpha}$. Then we set $L_{\alpha}=E_{\alpha}$. If there exists an extension $E_{\alpha} \subset E \in \mathfrak{V}$ such that $S_{\alpha}$ has a solution $(u_{1}, \dots, u_{n})$ in $E$, then we set $L_{\alpha}=\langle E_{\alpha},u_{1},\dots,u_{n} \rangle \subset E$. Since $\mathfrak{V}$ is closed under subalgebra, it follows that $L_{\alpha} \in \mathfrak{V}$ and we have 
 \[ |L_{\alpha}|=|E_{\alpha}| \leq \kappa .\]
 Now, we define 
 \[ H= \bigcup_{0 \leq \alpha \leq \kappa} L_{\alpha} \]
 which is an element of $\mathfrak{V}$. We have $|H|\leq {\kappa}^{2}=\kappa$, and hence 
 \[ max\{dim H, |K| \} \leq max \{ \aleph_{0}, dim L, |K| \},\]
therefore, we have $dim H \leq max \{ \aleph_{0}, dim L, |K| \}.$ 
\end{proof}

\begin{theorem}\cite{S2}
Let $\mathfrak{V}$ be an inductive class of algebras over field $K$. Suppose $\mathfrak{V}$ is closed under subalgebra and $L \in \mathfrak{V}$. Then there exists an algebra $L^\ast \in \mathfrak{V}$ with the following properties,
\begin{itemize}
    \item[1-] $L$ is a subalgebra of $L ^ \ast$.
    \item[2-] $L^\ast$ is existentially closed in the class $\mathfrak{V}$.
    \item[3-] $dim L^\ast \leq \{\aleph_0, dim L, |K| \}.$
\end{itemize}
\end{theorem}
\begin{proof}
 Let $H^{0}=L$ and $H^{1}=H$ be an algebra satisfying requirements of the previous lemma. Suppose $H^{m}$ is defined and let $H^{m+1}$ be an algebra obtained by the lemma from $H^{m}$. Then
 \[ dim H^{m+1} \leq max \{ \aleph_{0}, dim H^{m}, |K| \} = max \{ \aleph_{0}, dim L , |K|\}. \]
 Now, put
 \[ L^{\ast} = \bigcup_{m} H^{m}. \]
 Therefore, $L^{\ast}$ is an algebra which has the properties (1)-(3).
\end{proof}
We recall the notion of biderivation of Leibniz algebras  which has already been introduced in \cite{C1}. 
\begin{definition} 
Let $L$ be a Leibniz algebra. A biderivation of $L$ is a pair $(d,D)$ of $K$-linear maps $d,D \colon L \to L$ such that
\begin{equation}
    d([l,l^{\prime}])=[d(l),l^{\prime}]+[l,d(l^{\prime})],
\end{equation}
\begin{equation}
    D([l,l^{\prime}])=[D(l),l^{\prime}]-[D(l)^{\prime},l],
\end{equation}
\begin{equation}
    [l,d(l^{\prime})]=[l,D(l^{\prime})]
\end{equation}
for all $l,l^{\prime} \in L.$
\end{definition}
The set of all biderivations of $L$ is denoted by $Bider(L)$ which is a Leibniz algebra with the Leibniz bracket given by
\[ [(d_1,D_1),(d_2,D_2)]=(d_1d_2-d_2d_1,D_1d_2-d_2D_1).\]
As a quick example, let $l\in L$, then the pair $(ad(l),Ad(l))$ with $ad(l)(l^{\prime})=-[l^{\prime},l]$ and $Ad(l)(l^{\prime})=[l,l^{\prime}]$ for all $l^{\prime} \in L$, is a biderivation and $(ad(l),Ad(l))$ is called \emph{inner biderivation} of $L$. We use this concept during the proof of the next theorem.
\\On the basis of the properties of HNN-extensions of Leibniz algebras, we can provide the following embedding theorem. The proof of the theorem is similar to the case of Lie algebras.
\begin{theorem}\label{theorem1}
Let $L$ be a Leibniz algebra over field $K$. Then there exists a Leibniz algebra $L^{\ast}$ having the following properties:
\begin{itemize}
\item[1.] $L$ is a subalgebra of $L^{\ast}$.
\item[2.] For any nonzero $a,b, b^{\prime} \in L^{\ast}$, there exists $x,y \in L^{\ast}$ such that $[x,a]=b$ and $[a,y]=b^{\prime}$, and so $L^{\ast}$ is solvable. 
\item[3.] $dim L^{\ast} \leq max \{ {\aleph_{0}, dim L, |K}| \}.$
\item[4.] $L^{\ast}$ is not finitely generated.
\item[5.] Every finite-dimensional simple Leibniz algebra over field $K$ embeds in $L^{\ast}$.
\item[6.] If $K$ is finite and $A$ is finite-dimensional Leibniz algebra over $K$, then we have
\[ Bider(A) \cong \frac{N_{L^{\ast}(A)}}{C_{L^{\ast}(A)}} \]
\end{itemize}
\end{theorem}
\begin{proof}
 Let $\mathfrak{V}$ be the class of all Leibniz algebras. Theorem \ref{theorem1} implies that there exists an existentially closed Leibniz algebra $L^{\ast}$ containing $L$ such that 
 \[ dim L^{\ast} \leq max \{ \aleph_{0}, dim L, |K| \}. \] 
 Let $0 \neq a,b \in L^{\ast}$. Let $d~ \text{and}~ d^{\prime}: \langle a \rangle \to L$ be a derivation and an anti-derivation, respectively, and $d(a)=b$ and $d^{\prime}(a)=b^{\prime}$. Let consider HNN-extensions \ref{hnnderivation} and \ref{hnnantiderivation} of Leibniz algbera $L$. Then the  embeddability theorem \ref{hnn-extension} implies that $L$ embeds in both HNN-extensions. Therefore, the equations $[x,a]=b$ and $[a,x]=b^{\prime}$ have solutions in $L^{\ast}_{d}$ and $L^{\ast}_{d^{\prime}}$, respectively, so $2$ is proved. 
 \\Let suppose $x_1,\dots,x_n$ be a finite set of elements of Leibniz algebra $L^{\ast}$ and consider the following systems of equations 
 \[ [x,x_i]=0, ~ [x_i,x]=0, \]  where $1\leq i \leq n, ~ x \neq 0.$
 These systems have solutions in  the Leibniz algebra $L^{\ast} \times \langle x \rangle$, and so we have $C_{L^{\ast}} (\langle x_1,\dots,x_{n} \rangle ) \neq 0$. Therefore, $L^{\ast}$ is not finitely generated.
 \\Suppose $H$ is a finite-dimensional simple Leibniz algebra with basis $u_1,\dots,u_n$ with $[u_i,u_j]=\sum_{r} \lambda_{ij}^{r} u_{r}$. Let consider the system 
 \[ [x_i,x_j]=\sum_{r} \lambda_{ij}^{r} x_{r} \]
 for $1 \leq i,j \leq n$, $x_{i} \neq 0$ where $1 \leq i \leq n$. This system has a solution in $L^{\ast} \times H$ and so there is a nonzero homomorphism $H \to L^{\ast}$ and $H$ embeds in $L^{\ast}$.
 
To prove $6$, let $K$ be finite and $A$ be finite-dimensional subalgebra of $L^{\ast}$. Let $d$ and $d^{\prime}$ be derivation and anti-derivation maps, respectively. Let consider HNN-extensions corresponding to both derivation and anti-derivation 
\[ L_d^{\ast} := \langle L, t: d(a)=[a,t], \ a\in A \rangle,\]
and 
\[ L_{d^{\prime}}^{\ast} := \langle L, t: d^{\prime}(a)=[t,a], \ a\in A \rangle ,\]
in which the system $[a,x]=d(a)$ and $[y,a]=d^{\prime}(a)$ have solutions. Therefore, there are $x,y \in L^{\ast}$ such that $d(a)=[a,x]$ and $d^{\prime}(a)=[y,a]$ for all $a \in A$, so $x$ is in the left normalizer and $y$ is in the right normalizer. Let $N_{L^{\ast}}(A)$ be the normalizer of $A$ which is the intersection of left and right normalizer. Therefore, there is an epimorphism $N_{L^{\ast}}(A) \to Bider(L)$ with the kernel $C_{L^{\ast}}(A)$ and we have 
\[ Bider(A) \cong \frac{N_{L^{\ast}(A)}}{C_{L^{\ast}(A)}}. \]
\end{proof}

\section*{Acknowledgments}
The author was supported by CAPES postdoctoral scholarship (88887 318997/2019-00) and CNPq (152453/2019-9).

\bibliographystyle{amsplain}

\end{document}